\documentclass[a4paper,12pt]{amsart}

\usepackage[utf8]{inputenc}

\usepackage{geometry}
\usepackage{geometry}
\usepackage{amsmath,amsfonts,amsthm, amssymb,mathtools}
\usepackage{hyperref, color}

\usepackage{soul}

\newcommand{\rd}{\mathbb{R}^{d}}
\newcommand{\R}{\mathbb{R}}
\newcommand{\C}{\mathbb{C}}

\newcommand{\hA}{\widehat{A}}
\newcommand{\AhA}{A\times\widehat{A}}

\let\originalleft\left
\let\originalright\right
\renewcommand{\left}{\mathopen{}\mathclose\bgroup\originalleft}
\renewcommand{\right}{\aftergroup\egroup\originalright}

\newtheorem{theorem}{Theorem}[section]

\newtheorem{proposition}[theorem]{Proposition}
\newtheorem{remark}[theorem]{Remark}
\newtheorem{definition}[theorem]{Definition}
\newtheorem{corollary}[theorem]{Corollary}
\newtheorem{example}[theorem]{Example}
\numberwithin{equation}{section}

\title[Maximally localized  Gabor orthonormal bases]{Maximally localized Gabor orthonormal bases on locally compact Abelian groups}

\author[F. Nicola]{Fabio Nicola}
\address{Dipartimento di Scienze Matematiche, Politecnico di Torino, Corso Duca degli Abruzzi 24, 10129 Torino, Italy}
\email{fabio.nicola@polito.it}

\date{}

\begin{document} 

\begin{abstract}
A Gabor orthonormal basis, on a locally compact Abelian (LCA) group $A$,  is an orthonormal basis of $L^2(A)$ which consists of time-frequency shifts of some template $f\in L^2(A)$. It is well-known that, on $\rd$, the elements of such a basis cannot have a good time-frequency localization. The picture is drastically different on LCA groups containing a compact open subgroup, where one can easily construct examples of Gabor orthonormal bases with $f$ maximally localized, in the sense that the ambiguity function of $f$ (i.e. the correlation of $f$ with its time-frequency shifts) has support of minimum measure, compatibly with the uncertainty principle. In this paper we find {\it all} the Gabor orthonormal bases with this extremal property. To this end, we identify all the functions in $L^2(A)$ which are maximally localized in the time-frequency space in the above sense -- an issue which is open even for finite Abelian groups. As a byproduct, on every LCA group containing a compact open subgroup, we exhibit the complete family of optimizers for Lieb's uncertainty inequality, and we also show previously unknown optimizers on a general LCA group. 
\end{abstract}


\maketitle

\section{Introduction}

Let $A$ be a locally compact Abelian (LCA) group, and denote by $\hA$ its dual group. We will write $\langle x, \xi\rangle$, with $x\in A$ and $\xi\in \hA$, for the corresponding duality; hence $|\langle x, \xi\rangle|=1$. 
For $x\in A$, $\xi\in\hA$ we define the translation and modulation operators $T_x$ and $M_\xi$ on $L^2(A)$, and the corresponding time-frequency shifts $\pi(x,\xi)$ as 
\begin{equation}\label{eq transl}
T_x f(y)=f(y-x),\ \  M_\xi f(y)=\langle y,\xi\rangle f(y),\ \  \pi(x,\xi) f(y)= M_\xi T_x f(y),
\end{equation}
where $y\in A$. 

A Gabor orthonormal basis on $A$ is an orthonormal basis of $L^2(A)$ of the type 
\begin{equation}\label{eq ort}
\mathcal{G}(f,\Gamma):=\{\pi(z) f:\ z\in\Gamma\}
\end{equation}
where $f\in L^2(A)$ and $\Gamma$ is a (not necessarily countable) subset of $\AhA$. The function $f$ is called {\it window}. In other terms, a Gabor orthonormal basis is an orthonormal basis in an orbit of the Schr\"odinger representation of the Heisenberg group associated to $A$.  

There is a considerable amount of work on the construction of Gabor orthonormal bases on $\R$; see \cite{agora,casazza99,dutkay14,gabardo,lai,li05,liu01,liu03,pinos} and also \cite{grochenig18} for far reaching generalizations on nilpotent Lie groups. The focus is mainly on windows that are characteristic functions of some compact subset, hence with a poor frequency localization. In fact, the celebrated Balian-Low theorem states, roughly speaking, that there are no Gabor orthonormal bases on $\R^d$, generated by a window with a good time-frequency localization (see e.g. the survey \cite{benedetto95}), and similar obstructions also apply on LCA groups which have no compact open subgroups \cite{enstad,kaniuth}; see also \cite{grochenig13} for the issue of the time-frequency localization of Riesz bases in $\R^d$. 

The situation is drastically different for LCA groups containing a compact, open subgroup. Indeed, if $H\subset A$ is such a subgroup, the function $f= |H|^{-1/2} \chi_H$ (where $|H|$ stands for the measure of $H$) generates an orthonormal basis as in \eqref{eq ort} if $\Gamma$ is any set of representatives of the cosets of $H\times H^\bot$ in $\AhA$ (namely, $\Gamma$  contains exactly one element of each coset) \cite{grochenig98,grochenig07}. This window $f$ is, in a sense, maximally localized in the time-frequency space. To explain this latter point properly, we need some terminology. 

For $f,g\in L^2(A)$, the short-time Fourier transform of $f$ with window $g$ is the complex-valued function on $\AhA$ given by
\begin{equation}\label{eq vgf}
V_g f(x,\xi)=\langle f, \pi(x,\xi)g \rangle_{L^2(A)}\qquad (x,\xi)\in\AhA.
\end{equation}

It is known that, if $\|g\|_{L^2(A)}=1$, then $V_g:L^2(A)\to L^2(\AhA)$ is an isometry, so that the quantity $|V_g f|^2$ can be regarded as a time-frequency energy density of $f$ (see \cite{grochenig98,grochenig_book,lieb_book}). More generally, the time-frequency localization of a function $f\in L^2(A)$ can be measured in terms of the $L^p$-norm of $V_f f$ -- the so-called {\it ambiguity function} -- and corresponding versions of the uncertainty principle, such as Lieb's uncertainty inequality, can be stated (see \cite{grochenig98,lieb} and Section \ref{sec conc} below). We recall, in particular, the following elementary result (see \cite{galbis,grochenig98,krahmer}, Theorem \ref{thm 2.2} below, and also \cite{bonami_demange_jaming,groechenig03} for other formulations of the uncertainty principle in terms of the short-time Fourier transform), which corresponds to a limiting case of an $L^p$-estimate as $p\to 0$ (cf. Remark \ref{rem pto0} below):
 
{\it 
 Let $f\in L^2(A)\setminus\{0\}$. Then } 
 \begin{equation}\label{eq uncert}
|\{z\in\AhA: V_f f(z)\not=0\}|\geq 1.
\end{equation}

Here we used the notation $|S|$ for the measure of a set $S\subset \AhA$, where the Haar measures on $A$ and $\hA$ are chosen so that the Plancherel formula holds true. The inequality \eqref{eq uncert} can be regarded as a time-frequency version of a lower bound, first proved by Matolcsi and Sz\H{u}cs in \cite{matolcsi73} (see also \cite{bonami,donoho, smith, tao,wigderson}) and usually referred to as the Donoho-Stark uncertainty principle, for the product of the measures of the supports of $f$ and $\widehat{f}$. 

%Notice that the expression in the left-hand side of \eqref{eq uncert} can be regarded as the limiting case of $\|V_f f\|_{L^p(\AhA)}^p$ as $p\to 0^+$. 
%In $\R^d$ the left-hand side of \eqref{eq uncert} is in fact infinite \cite{grochenig_book} but the heuristic principle that functions in $L^2(\R^d)\setminus\{0\}$ cannot be concentrated in a region of measure less that $1$ in the time-frequency (alias phase) space, is still valid and is ultimately responsible for the deepest existence and regularity results for PDEs \cite{fefferman}; see also \cite{lieb}. 
The inequality \eqref{eq uncert}  is sharp on every LCA group containing a compact, open subgroup $H$. Indeed, for the function $f=|H|^{-1/2}\chi_H$ considered above, we have $V_f f=\chi_{H\times H^\bot}$ (cf. \cite{grochenig98} and Proposition \ref{pro 4} below) and therefore $f$ is maximally localized in the time-frequency space, in the sense that the inequality \eqref{eq uncert} is saturated ($|H\times H^\bot|=1$ by the Plancherel formula).
%(cf. \cite[Formula (4.4.6)]{reiter00}). 
This is a desirable property that guarantees that the ``analysis operator" associated to the corresponding basis \eqref{eq ort}, that is $L^2(A)\ni h\mapsto V_f h(z)$, with $z\in \Gamma$, is able to resolve the blobs of energy of $h$ in the time-frequency space, with the highest possible resolution -- at least in the measure theoretic sense. 

This discussion raises the problem of identifying {\it all} the Gabor orthonormal bases generated by a function $f$ maximally localized in the above sense, namely such that the set where $V_f f\not=0$ has measure $1$ (this implies that the same extremal property holds for every element of the basis). This issue is also motivated, on one hand, by the recent advances \cite{enstad} on the Balian-Low theorem on general (second countable) LCA groups which have {\it no} compact open subgroups and, on the other hand, by the recent, considerable interest in optimizers of uncertainty inequalities on Euclidean spaces and Riemannian manifolds \cite{abreu,dias22,frank22,kalaj,knutsen23,kulikov,kulikov22bis,lieb21,romero,nicola22,joao_tilli}. In a sense, the results in this paper can be regarded as complementary to the no-go results in \cite{enstad}.

The main problem therefore lies in the identification of all the functions $f\in L^2(A)$ for which equality occurs in \eqref{eq uncert}, which is an open issue even for finite Abelian groups. This should not come as a surprise, in light of other extremal problems, e.g. for the Young and Hausdorff-Young inequalities, which are notoriously difficult even on LCA groups containing a compact, open subgroup (see \cite[Theorem 3]{fournier77} and \cite[Section 43]{hewitt70}). Indeed, for the finite cyclic group $\mathbb{Z}_N$ all the optimizers for the inequality \eqref{eq uncert} were obtained only recently, in \cite{nicola23}; see also \cite{galbis} for a particular case.  

Our first result (Theorem \ref{mainthm 1}) provides the complete answer to this problem and can be summarized as follows:

{\it Equality occurs in \eqref{eq uncert} if and only if $f=c T_{x_0} h$, for some $c\in\mathbb{C}\setminus\{0\}$, $x_0\in A$ and some subcharacter of second degree $h$ of $A$. In this case,  $\|f\|_{L^2(A)}^{-2}{V_f f}$ is a subcharacter of second degree of $\AhA$, and its support is a maximal compact open isotropic subgroup of $\AhA$.
}
 
A subcharacter of second degree of $A$ is a continuous function $h:A\to\mathbb{C}$ such that, for some compact open subgroup $H\subset A$, the restriction of $h$ to $H$ is a character of second degree of $H$ in the sense of Weil \cite{weil64} (see Section \ref{sec char} below) and $h(x)=0$ for $x\in 
A\setminus H$. The term ``isotropic" refers to the standard symplectic structure of $\AhA$ (see Section \ref{sec sec3}). Hence, maximal compact open isotropic subgroups of $\AhA$ are the minimum uncertainty phase-space cells -- the so-called ``quantum blobs" -- and play the role of the symplectic images of the unit ball (or box) in $\R^d\times\R^d$ \cite{degosson09,fefferman}. Also, notice the formal analogy with the extremal problem for the Hausdorff-Young inequality (on LCA groups), where the optimizers are the constant multiples  of translates of subcharacters \cite[Theorem 43.13]{hewitt70}. 

The study of the cases of equality in \eqref{eq uncert} on $\mathbb{Z}_N$ \cite{nicola23} relies on the explicit description of the subgroups of order $N$ of  $\mathbb{Z}_N\times \mathbb{Z}_N$ and the equally explicit construction of ``finite chirps", which is available in that case (see \cite{casazza06,feichtinger2}). In the present generality we have to follow a more conceptual pattern. To give a flavour of the argument let us briefly outline how the above mentioned subcharacter of second degree arises. By using the covariance property for the short-time Fourier transform we will show that, if the set $G=\{z\in \AhA: V_f f(z)
\not=0\}$ has measure $1$, then $G$ is a maximal compact open {\it isotropic} subgroup of $\AhA$. The projection on the first factor allows us to regard $G$ as an extension of a compact open subgroup $H$ of $A$ by $H^\bot\subset \hA$. This induces, in a natural way, a continuous {\it symmetric} homomorphism $H\to\widehat{H}$, to which we can associate the desired character $h$ of second degree of $H$ -- here we use an extension theorem for characters of second degree from \cite{igusa68}. The analogous problem for the short-time Fourier transform $V_g f$ is addressed in Theorem \ref{mainthm 2}. 

Let us now come back to the above problem of maximally localized Gabor orthonormal bases. We anticipate here the main result (Corollary \ref{mainthm 4}), which provides the desired, full characterization. 

{\it Let $f\in L^2(A)$, $\|f\|_{L^2(A)}=1$ and $G:=\{z\in\AhA: V_f f(z)\not=0\}$. Let $\Gamma\subset\AhA$. 
 Then $\mathcal{G}(f,\Gamma)$ (cf. \eqref{eq ort}) is an orthonormal basis of $L^2(A)$ and $|G|=1$ if and only if $f=cT_{x_0} h$ for some $c\in\mathbb{C}$, $x_0\in A$ and some subcharacter $h$ of second degree of $A$, and $\Gamma$ is a set of representatives of the cosets of $G$ in $\AhA$.
}

The subsets of $\AhA$ defined by $V_f(\pi(z) f)\not=0$, with $z\in\Gamma$, are precisely the cosets of the maximal compact open isotropic subgroup $G$, and define therefore a tiling of $\AhA$. In fact, it turns out that for every tiling of this type there is a corresponding Gabor orthonormal basis (see Remark \ref{rem 6.3}). But the main point of the above result is clearly represented by the necessary condition. We refer the reader to \cite{iosevich,zhou} and the references therein for other constructions of (not necessarily maximally localized) Gabor orthonormal bases in the setting of finite Abelian groups and to \cite{enstad} for general LCA groups.  

In Section \ref{sec conc} we will show  that, on any LCA group containing a compact open subgroup, every extremal function for the uncertainty inequality \eqref{eq uncert} is also an extremal function for Lieb's uncertainty inequality \cite{grochenig98,lieb} and vice versa. This will also provide, as a byproduct, previously unknown optimizers for Lieb's inequality on every LCA group, because such a group is topologically isomorphic to $\R^d\times A_0$ for some integer $d\geq 0$ and some LCA group $A_0$ containing a compact open subgroup. We postpone to a subsequent work, in preparation, the problem of identifying {\it all} the optimizers for Lieb's uncertainty inequality on a general LCA group. We also refer the reader to \cite{frank22, kulikov22bis} and the references therein for recent advances on Lieb's inequality on Euclidean spaces and Riemannian manifolds. 

%We also address the reader to \cite{dias22,frank22,lieb21,kulikov,kulikov22bis,nicola22} and the references therein for recent advances on Lieb's uncertainty inequality on Euclidean spaces, the Poincar\'e disk, and the Riemann sphere.  

\section{Notation and preliminary results}\label{sec prel}
\subsection{Notation}
We use the notation introduced at the beginning of the introduction. Hence $A$ denotes a locally compact Abelian group, and $\hA$ its dual group. The pairing $\AhA\to\mathbb{T}$ (the multiplicative group of complex numbers of modulus $1$) given by $(x,\xi)\to \langle x,\xi\rangle$ is therefore well defined. We will write the group laws in $A$ and $\hA$ additively, hence $\langle x+y,\xi\rangle=\langle x,\xi\rangle \langle y,\xi\rangle$ and $\langle x,\xi+\eta\rangle =\langle x,\xi\rangle \langle x,\eta\rangle$. The groups $A,\hA$  are equipped with Haar measures related so that the Plancherel formula holds true. The Haar measure on $\AhA$ is given by the product measure. The inner product in $L^2(A)$ is denoted by $\langle\cdot,\cdot\rangle_{L^2(A)}$. We denote by $|S|$ the measure of a subset $S$ (of $A$ or $\hA$, or $\AhA$), and by $\chi_S$ its characteristic function. 

We refer to  \eqref{eq transl} for the definition of the translation operators $T_x$, $x\in A$, the modulation operators $M_\xi$, $\xi\in\hA$, and the phase-space shifts $\pi(x,\xi)= M_\xi T_x$. They are unitary operators on $L^2(A)$. The short-time Fourier transform $V_g f$, for $f,g\in\ L^2(A)$, was defined in \eqref{eq vgf}; more explicitly 
\begin{align}\label{eq vgf2}
V_g f(x,\xi)&=\langle f, \pi(x,\xi)g \rangle_{L^2(A)}\\
&=\int_A \overline{\langle y,\xi\rangle} f(y) \overline{g(y -x)}\, dy\nonumber\qquad x\in A, \ \xi\in\hA.
\end{align}

\subsection{Preliminaries from time-frequency analysis} 
We recall some basic results about time-frequency analysis on LCA groups. We refer the reader to \cite{feichtinger,grochenig98} for details (see also \cite{grochenig_book} for the analogous results in $\R^d$). 

The following result generalizes a well-known property of the short-time Fourier transform in $\R^d$ \cite{grochenig_book}. 
\begin{proposition}\label{pro 2.1}
Let $f,g\in L^2(A)$.
Then $V_g f$ is a continuous function on $\AhA$, which vanishes at infinity.
\end{proposition}
\begin{proof}
Since the unitary representations $A\ni x\mapsto T_x$ and $\hA\ni\xi\mapsto M_\xi$ are strongly continuous on $L^2(A)$, $V_g f$ is a continuous function on $\AhA$. Moreover, by the definition of $V_g f$ in \eqref{eq vgf2} we have $|V_g f(x,\xi)|\leq |f|\ast |\tilde{g}|(x)$ and similarly,  $|V_g f(x,\xi)|\leq |\widehat{f}|\ast |\tilde{\widehat{g}}|(\xi)$, with $\tilde{g}(y):=g(-y)$. On the other hand, functions in $L^2\ast L^2$ tend to zero at infinity.
\end{proof}
As a consequence, we also see that the set $\{z\in \AhA:\ V_g f(z)\not=0\}$ is $\sigma$-compact. 

We will need a few elementary formulas concerning time-frequency shifts and the short-time Fourier transform. 

First, for $x,y\in A$, $\xi,\eta\in \hA$ we have the commutation relations
\begin{equation}\label{eq 2.2}
\pi(x,\xi)\pi(y,\eta)=\langle y,\xi\rangle\overline{\langle x,\eta\rangle}\pi(y,\eta)\pi(x,\xi).
\end{equation}
As a consequence, 
the following covariance-type properties hold true, for $x,y\in A$, $\xi,\eta\in\hA$:  
\begin{equation}\label{eq 2.3}
V_{g}(\pi(x,\xi)f)(y,\eta)= \langle x,\xi\rangle\overline{\langle x,\eta\rangle} V_g f(y-x,\eta-\xi)
\end{equation}
and
\begin{equation}\label{eq 2.4}
V_{\pi(x,\xi)g}(\pi(x,\xi)f)(y,\eta)= \langle y,\xi\rangle \overline{\langle x,\eta\rangle}  V_g f(y,\eta).
\end{equation}

An application of the Cauchy-Schwarz inequality gives at once the following pointwise estimate:
\begin{equation}\label{eq cs}
|V_g f(x,\xi)|\leq \|f\|_{L^2(A)} \|g\|_{L^2(A)}\qquad x\in A,\ \xi\in\hA.
\end{equation}
We finally recall the
 Parseval equality
\begin{equation}\label{eq energy}
\|V_g f\|_{L^2(\AhA)}=\|f\|_{L^2(A)} \|g\|_{L^2(A)}.
\end{equation}
The following uncertainty inequality was first proved in \cite{krahmer} in the case of finite Abelian groups; see also \cite{galbis,grochenig98,nicola23}.

\begin{theorem}\label{thm 2.2} Let $f,g\in L^2(A)\setminus\{0\}$ and $S=\{z\in\AhA: V_g f(z)\not=0\}$. We have $|S|\geq 1$. If $|S|=1$ then $|V_g f|=c \chi_S$, with $c=\|f\|_{L^2(A)}\|g\|_{L^2(A)}$, and therefore $S$ is compact and open. 
\end{theorem}
\begin{proof}
We can suppose $\|f\|_{L^2(A)}=\|g\|_{L^2(A)}=1$.
    By \eqref{eq energy} and \eqref{eq cs}, we have 
    \[
1= \int_{S} |V_g f(x,\xi)|^2 dx\, d\xi  \leq  \int_S dx\, d\xi =|S|.
    \]
    If equality occurs in the above inequality, then $|V_g f(z)|=1$ for almost every $z\in S$, and therefore for every $z\in S$, since $V_g f$ is continuous (Proposition \ref{pro 2.1}), and $S$ is open. Hence  $|V_g f|=\chi_S$, which implies that  $S$ is also closed. Moreover, since  $V_g f$ tends to zero at infinity, $S$ is contained in a compact subset, and therefore is compact. 
\end{proof}

We will also need the following uniqueness result for the ambiguity function $V_f f$. 
\begin{proposition}\label{pro 2.3}
Let $f,g\in L^2(A)$. Then $V_f f=V_g g$ on $\AhA$ if and only if  $f=c g$ for some $c\in\C$, $|c|=1$.
\end{proposition}
\begin{proof}
For the sake of completeness we provide a proof which is similar (but not exactly equal) to that in \cite[Section 4.2]{grochenig_book} for $A=\R^d$. 

For every $x\in A$, we can regard $V_f f(x,\cdot)$ as the Fourier transform of the $L^1$ function $f\overline{T_x f}$. Hence, if $V_f f=V_g g$, by the Fourier uniqueness theorem, for every $x\in A$ we have 
\[
f(y) \overline{f(y-x)}=g(y) \overline{g(y-x)}
\]
for almost every $y\in A$. By the Fubini theorem (both sides vanish on the complement of a $\sigma$-compact set in $A\times A$; cf. \cite[page 44]{follandbook}), the above equality holds for almost every $(x,y)\in A\times A$. Multiplying by $f(y-x)$ and integrating with respect to $x$ yields $ \|f\|^2_{L^2(A)} f=\langle f,g\rangle_{L^2(A)}g$, which gives the desired conclusion. 

The converse implication is obvious. 
\end{proof}

\subsection{Characters of second degree}\label{sec char} We recall from \cite{weil64} (see also \cite{reiter89}) the  notion of character of second degree. 
\begin{definition}\label{def charsecond}
    Given a continuous symmetric homomorphism $\phi:A\to \hA$ (hence $\langle x,\phi(y)\rangle=\langle y,\phi(x)\rangle$ for $x,y\in A$) a character of second degree of $A$, associated to $\phi$, is a continuous function $f:A\to\mathbb{T}$ such that 
\[
f(x+y)=f(x)f(y)\langle x,\phi(y)\rangle \qquad x,y\in A.
\]
\end{definition}
The following result provides the existence and uniqueness, up to multiplication by characters, of characters of second degree associated to a given homomorphism $\phi$ as above. 
\begin{theorem}\label{thm reiter} 
Given a continuous symmetric homomorphism $\phi:A\to\hA$ there exists a character of second degree associated to $\phi$. Two characters of second degree associated to the same $\phi$ differs by the multiplication by a character. 
\end{theorem}
The uniqueness is an immediate consequence of the definition. The existence was first proved in \cite[Lemma 6]{igusa68} (see also \cite[page 37]{reiter89} for an easier proof due to M. Burger, and \cite[Theorem 2.3]{baggett73} for generalizations). Explicit constructions are available in particular cases; for example, if multiplication by $2$ is an automorphism of $G$ then one can take $f(x)=\langle x,\phi(2^{-1}x)\rangle$ \cite[page 146]{weil64}. We refer the reader to \cite[Section 7.7]{reiter89} for an explicit construction when $G$ a finite dimensional vector space over a local field, 
\cite{feichtinger2} for  $\mathbb{Z}_N$, and \cite{kaiblinger09} for finite Abelian groups.

\section{Some symplectic analysis on $\AhA$}\label{sec sec3}
In this section we prove some auxiliary results concerning subgroups of $\AhA$, in connection with the standard symplectic structure of $\AhA$, that is the bicharacter $\sigma:(\AhA)\times (\AhA)\to\mathbb{T}$ given by (cf. \eqref{eq 2.2})
\begin{equation}\label{defsigma}
\sigma((x,\xi),(y,\eta))= \langle y,\xi\rangle\overline{\langle x,\eta\rangle}
\qquad (x,\xi),\ (y,\eta)\in \AhA.  \end{equation}
The following description of the compact open subgroups of $\AhA$ will be crucial in the following. 
\begin{proposition}\label{pro 1}
Let $H\subset 
 A$, $K\subset \hA$ be compact open subgroups, and let $\phi:H\to \hA/K$ be a continuous homomorphism. Then the set 
\begin{equation}\label{eq subgroup}
     G=\{(x,\xi)\in\AhA:\ x\in H,\ \xi\in \phi(x)\}
 \end{equation}
 is a compact open subgroup of $\AhA$, and $|G|=|H||K|$. 
 
 Every compact open subgroup of $\AhA$ arises in this way for a unique triple $(H,K,\phi)$ as above.
\end{proposition}
\begin{proof}
It is easy to see that the set $G$ in \eqref{eq subgroup} is indeed a subgroup of $\AhA$. Since $H$ and $K$ are compact and $\phi$ is continuous it follows from some general (non-trivial) results from the theory of topological groups (cf. \cite[Note 5.24]{hewitt63}), that $G$ is compact. However, since $H$ and $K$ are open, we can apply a more direct argument, that also has the advantage to give some more information, namely that $G$ is in fact a finite union of pairwise disjoint compact open subsets of $\AhA$ of product type. Precisely, since $K$ is open, $\hA/K$ is discrete, and $\phi(H)\subset \hA/K$ is compact, therefore finite. On the other hand we have 
\[
G=\cup_{B\in \phi(H)}\, \phi^{-1}(\{B\})\times B.
\]
The sets $B\subset\hA$ above are compact and open, because they are cosets of $K$ in $\hA$. The subsets $\phi^{-1}(\{B\})\subset H$ are open and closed and therefore compact, because $H$ is compact. This shows that $G$ is compact and open. 

For every $x\in H$, $\phi(x)$ is a coset of $K$ in $\hA$, and therefore has the same measure as $K$. Hence   
\[
|G|=\int_H \int_{\phi(x)} \,d\xi\, dx= |H||K|.
\]

Now, given any compact open subgroup $G\subset\AhA$, let $\pi_1:G\to A$ be the projection on the first factor, and set $H=\pi_1(G)$ and $K=G\cap \hA={\rm Ker}\, \pi_1$ (where $\hA$ is regarded as a subgroup of $\AhA$). It is clear that $K$ is a compact and open subgroup of $\hA$ and $H$ is a compact open subgroup of $A$, because the projection $A\times\hA\to A$ is an open map. Hence we have the short exact sequence 
\[
0\to K\to G\to H\to 0.
\]
The map $G\to H$ is open, and therefore the algebraic isomorphism 
 $\beta:G/K\to H$ is in fact a topological isomorphism (\cite[Theorem 5.27]{hewitt63}). Now set $\phi=\pi_2'\circ \beta^{-1}$, where $\pi_2': G/K\to \hA/K$ is the natural epimorphism induced by the projection $\pi_2:G\to \hA$ on the quotient spaces. Then $\phi:H\to\hA/K$ is a continuous homomorphism, and clearly \eqref{eq subgroup} holds for this triple $(H,K,\phi)$. 

Finally, it is clear from \eqref{eq subgroup} that the triple $(H,K,\phi)$ is uniquely determined by $G$.
\end{proof}
\begin{remark}
   It is easy to see that $G$ and $H\times K$ in Proposition \ref{pro 1} are homeomorphic. Indeed, with the notation of the above proof, choosing a representative $\xi$ out of any $[\xi]\in \phi(H)$ ($\phi(H)$ is a finite set), yields a continuous section $\alpha:\phi(H)\to \hA$ and therefore a lifting $\tilde{\phi}:=\alpha \circ \phi: H\to \hA$ of $\phi$. We now extend $\tilde{\phi}$ to $A$ by setting $\tilde{\phi}(x)=0$ for $x\in A\setminus H$. Since $H$ is both open and closed, $\tilde{\phi}:A\to\hA$ is continuous, and the map $\AhA\to\AhA$ given by $(x,\xi)\mapsto (x, \tilde{\phi}(x)+\xi)$ is a homeomorphism (although not a group homomorphism) which maps $H\times K$ onto $G$.  
\end{remark}
We now single out a class of subgroups and provide a convenient characterization in the spirit of Proposition \ref{pro 1}.   
\begin{definition}
A subgroup $G\subset \AhA$ is called {\rm isotropic} if $\sigma(z,w)=1$ (cf. \eqref{defsigma}) for every $z,w\in\AhA$.
\end{definition}

Consider a triple $(H,K,\phi)$ as in Proposition \ref{pro 1} and assume, in addition, that $K\subset H^\bot$. Then  a natural epimorphism $\hA/K\to \hA/H^\bot\simeq \widehat{H}$ is induced. Hence, for $x\in H$, we can regard $\phi(x)\in \hA/K$ as a character of $H$, whose value at $y\in A$ will be denoted by $\langle y,\phi(x)\rangle$. In concrete terms, \[
\langle y,\phi(x)\rangle:= \langle y,\xi\rangle\quad \textrm{for any}\quad \xi\in \phi(x).
\] 
\begin{definition}
If $K\subset H^\bot$, a continuous homomorphism $\phi:H\to \hA/K$ is said  symmetric if 
\[
\langle x,\phi(y)\rangle= \langle y,\phi(x)\rangle\qquad x,y\in H.
\]
\end{definition}
We recall, for future reference, that if $H\subset A$ is a compact open subgroup, then $H^\bot$ is a compact open subgroup of $\hA$ (see e.g. \cite[Lemma 6.2.3 (b)]{grochenig98}). 

\begin{proposition}\label{pro 2}
Let $G\subset\AhA$ be a compact open subgroup of $\AhA$ and let $(H,K,\phi)$ be the corresponding triple (cf. Proposition \ref{pro 1}). Then $G$ is isotropic if and only if $K\subset H^\bot$ and $\phi:H\to \hA/K$ is symmetric.
Moreover $|G|\leq 1.$
\end{proposition}
\begin{proof}
Let $G$ be isotropic. Let $x\in H$, $\xi\in \phi(x)$ and $\eta\in K$, so that $z=(x,\xi)\in G$ and $w=(x,\xi+\eta)\in G$. We have 
\[
1=\sigma(z,w)=\langle x,\xi\rangle\overline{\langle x,\xi+\eta\rangle}=\overline{\langle x, \eta\rangle},
\]
and therefore $K\subset H^\bot$. Now, if $x,y\in H$ and $\xi\in \phi(x)$, $\eta\in \phi(y)$, so that  $z=(x,\xi)\in G$ and $w=(y,\eta)\in G$, we have $1=\sigma(z,w)=\langle y,\xi\rangle\overline{\langle x,\eta\rangle}$, which means that $\phi$ is symmetric. 

Vice versa, it is clear from the above computation that if $K\subset H^\bot$ and $\phi$ is symmetric then $G$ is isotropic. 

Finally, by Proposition \ref{pro 1},  $|G|=|H||K|\leq |H||H^\bot|=1$, where the last equality follows from the Plancherel formula (see \cite[Formula (4.4.6)]{reiter00}).
\end{proof}
We finally characterize the compact open isotropic subgroups of maximum measure.

\begin{proposition}\label{pro 3}
    Let $G\subset\AhA$ be a compact open isotropic subgroup and let $(H,K,\phi)$ be the corresponding triple (cf. Propositions \ref{pro 1} and \ref{pro 2}). The following statements are equivalent: 

\begin{itemize}
\item[(a)] $|G|=1$. \medskip\par
\item[(b)] $K=H^\bot$.
\medskip\par
\item[(c)] $G$ is a maximal compact open isotropic subgroup, i.e. if $G'\subset\AhA$ is a compact open isotropic subgroup with $G\subset G'$ then $G'=G$.
\end{itemize}

\end{proposition}
\begin{proof}
(a) $\Rightarrow$ (b). We know from Proposition \ref{pro 2}  that $K\subset H^\bot$. If this inclusion were strict then $|K|<|H^\bot|$, because $|H^\bot|<\infty$ and $|H^\bot\setminus K|>0$, being $H^\bot\setminus K$ open. Hence, since $0<|H|<\infty$, by Proposition \ref{pro 1} $|G|=|H||K|<|H||H^\bot|= 1$, which is a contradiction.

(b) $\Rightarrow$ (c). Let $G'\subset\AhA$ be a compact open isotropic subgroup with $G\subset G'$. Let $(H', K', \phi')$ be the corresponding triple, as in Propositions \ref{pro 1} and \ref{pro 2}. Since $G\subset G'$ we have $H\subset H'$ and $K\subset K'$. On the other hand, by Proposition \ref{pro 2} and the assumption, we have $H'\subset K'^\bot\subset K^\bot=H$. Hence $H=H'$ and $K=K'$. Since, for $x\in H$, $\phi(x)$ and $\phi'(x)$ are two cosets of $K$ in $\hA$ and $\phi(x)\subset \phi'(x)$ we have $\phi(x)=\phi'(x)$ and therefore $G'=G$. 

(c) $\Rightarrow$ (a). Suppose, by contradiction, that $|G|=|H||K|<1$. Then the inclusion $K\subset H^\bot$ is strict. Consider the subgroup $G'\subset\AhA$ associated to the triple $(H,H^\bot,\phi')$, with $\phi'=\alpha\circ\phi:H\to\hA/H^\bot$, where $\alpha:\hA/K\to\hA/H^\bot$ is the natural epimorphism. Since $\phi$ is symmetric, the same holds for  $\phi'$. Hence, by Proposition \ref{pro 2}, $G'$ is a compact open isotropic subgroup of $\AhA$ and $G\subset G'$ with strict inclusion, because $G\cap\hA=K\subset H^\bot=G'\cap\hA$ strictly. This is a contradiction. 
\end{proof}
\begin{corollary}\label{rem max}
Every compact open isotropic subgroup of $\AhA$ is contained in a maximal compact open isotropic subgroup. 
\end{corollary}
\begin{proof}
Let $G\subset\AhA$ be a compact open isotropic subgroup and $(H,K,\phi)$ be its associated triple (Propositions \ref{pro 1} and \ref{pro 2}), hence $K\subset H^\bot$ and $\phi$ is symmetric. Then the compact open isotropic subgroup associated to the triple $(H,H^\bot,\phi')$, with $\phi':H\to \widehat{H}$ as in the proof of ``(c) $\Rightarrow$ (a)" in Proposition \ref{pro 3}, is maximal by Proposition \ref{pro 3} and contains $G$.
\end{proof}

\begin{remark}
Notice that the result of the above corollary is no longer valid if we drop the adjective ``isotropic". For example, consider the $p$-adic field $\mathbb{Q}_p$. Its topological dual $\widehat{\mathbb{Q}}_p$ can be identified with $\mathbb{Q}_p$. The balls $B_j:=\{x\in \mathbb{Q}_p:\ |x|_p\leq p^j\}$, $j\in\mathbb{Z}$, are compact open subgroups of $\mathbb{Q}_p$. Hence the sets $B_j\times B_j$, are compact open subgroups of $\mathbb{Q}_p\times \widehat{\mathbb{Q}}_p$. However, if $K\subset \mathbb{Q}_p\times \widehat{\mathbb{Q}}_p$ is a compact subset, then $K\subset B_j\times B_j$ for some $j$, and $K\subset B_{j+1}\times B_{j+1}$ strictly. We refer to \cite{enstad19} for a quick review of the $p$-adic number system from the perspective of time-frequency analysis. 
\end{remark}

We conclude this section with a result which will be useful in the following. 
\begin{proposition}\label{lem} 
Let $G\subset\AhA$ be a maximal compact open isotropic subgroup, hence $|G|=1$ (cf. Proposition \ref{pro 3}). Let $(H,H^\bot,\phi)$ be the corresponding triple (cf. Propositions \ref{pro 1}, \ref{pro 2} and \ref{pro 3}). If $f$ is a character of second degree of $H$ associated to $\phi$ (cf. Definition \ref{def charsecond}), then the function $G\to\mathbb{T}$ given by 
\[
(x,\xi)\mapsto\overline{f(x)}\qquad (x,\xi)\in G
\]
 is a character of second degree of $G$ associated to the continuous symmetric homomorphism $\phi':G\to\widehat{G}$  given by 
    \begin{equation}\label{eq phi'}
    \langle (x,\xi),\phi'(y,\eta)\rangle=\overline{\langle x,\eta\rangle} \qquad (x,\xi),(y,\eta)\in G. 
    \end{equation}
\end{proposition}
\begin{proof}
For $(x,\xi),(y,\eta)\in G$ we have 
\[
\overline{f(x+y)}=\overline{f(x)}\,\overline{f(y)}\,\overline{\langle x,\phi(y)\rangle}= \overline{f(x)}\,\overline{f(y)}\,\overline{\langle x,\eta\rangle}
\]
because $\eta\in\phi(y)$ (where $\phi(y)\in\widehat{H}\simeq \hA/H^\bot$ is now regarded as a coset of $H^\bot$ in $\hA$).
\end{proof}

\section{Optimizers for the ambiguity function}

For $f\in L^2(A)$, by \eqref{eq cs} we see that, for every $z\in\AhA$, \[
|V_f f(z)|\leq V_f f(0)=\|f\|_{L^2(A)}^2.
\]
The following result provides some properties of the set where $|V_f f|$ attains its maximum value.
\begin{proposition}\label{pro 4.1}
    Let $f\in L^2(A)\setminus\{0\}$, and 
    \[
    G=\{z\in\AhA:\ |V_f f(z)|=V_f f(0)\}.
    \]
    Then $G$ is a compact isotropic subgroup of $\AhA$ and the restriction of the function $\|f\|_{L^2(A)}^{-2} V_f f$ to $G$ is a character of second degree associated to the continuous symmetric homomorphism
    $
    \phi':G\to \widehat{G}$ given in \eqref{eq phi'}.
    
    Indeed, for $(x,\xi)\in G$, $(y,\eta)\in \AhA$ we have 
\begin{equation}\label{eq Vff char}
         V_f f(y+x,\eta+\xi)= \|f\|_{L^2(A)}^{-2}V_f f(x,\xi) V_f f(y,\eta) \overline{\langle x,\eta\rangle}.
\end{equation}
In particular, $|V_f f|$ is constant on every coset of $G$ in $\AhA$.
\end{proposition}
\begin{proof}
Without loss of generality we can suppose $\|f\|_{L^2(A)}=1$, hence $V_f f(0)=\|f\|_{L^2(A)}^2=1$.

Since $V_f f$ is a continuous function which vanishes at infinity (Proposition \ref{pro 2.1}), $G$ is compact.

 For $z\in G$ we have 
\[
|\langle f,\pi(z) f\rangle|=1,
\]
and therefore 
\[
\pi(z)f=c(z) f
\]
for some $c(z)\in\mathbb{C}$, $|c(z)|=1$. 
As a consequence, 
\[
V_f(\pi(z)f)f(w)=c(z) V_f f(w)
\]
for $w\in\AhA$. Setting $z=(x,\xi)$ and $w=(y,\eta)$, by \eqref{eq 2.3} we have 
\[
\langle x,\xi\rangle\overline{\langle x,\eta\rangle} V_f f(y-x,\eta-\xi)=c(x,\xi) V_f f(y,\eta). 
\]
Setting $y=0, \eta=0$, we obtain 
\[
\langle x,\xi\rangle V_f f(-x,-\xi)=c(x,\xi), 
\]
which gives 
\begin{equation}\label{eq Vff charbis}
         V_f f(y-x,\eta-\xi)= V_f f(-x,-\xi) V_f f(y,\eta) \langle x,\eta\rangle
        \end{equation}
for $(x,\xi)\in G$, $(y,\eta)\in\AhA$.

On the other hand, it is clear from the very definition \eqref{eq vgf2} of $V_f f$, that $|V_f f(-w)|=|V_f f(w)|$ for $w\in\AhA$, and therefore if $(x,\xi)\in G$ then $(-x,-\xi)\in G$ as well. 

Hence we obtain
\begin{equation}\label{eq Vff charter}
         V_f f(y+x,\eta+\xi)= V_f f(x,\xi) V_f f(y,\eta) \overline{\langle x,\eta\rangle}
        \end{equation}
for $(x,\xi)\in G$, $(y,\eta)\in\AhA$, which proves  \eqref{eq Vff char}.

By  \eqref{eq Vff charter}, if $(x,\xi), (y,\eta)\in G$ then $(x+y,\xi+\eta)\in G$, hence $G$ is a subgroup of $\AhA$ ($(0,0)\in G$, of course). Finally exchanging the roles of $(x,\xi)$ and $(y,\eta)$ in \eqref{eq Vff charter} yields that $G$ is isotropic. 
\end{proof}
\begin{remark} For $A=\R^d$ we recapture the well-known {\rm radar correlation estimate} \cite[Lemma 4.2.1]{grochenig_book}: if $f\in L^2(\R^d)\setminus\{0\}$ and $z\in\R^{2d}$, $z\not=0$, then $|V_f f(z)|<V_f f(0)$. 
\end{remark}

We now introduce the notion of {\it subcharacter of second degree}. This terminology -- in fact non-standard -- is inspired by the definition of subcharacter \cite[Definition 43.3]{hewitt70}, that is a character of a compact open subgroup $H\subset A$, extended by $0$ on $A\setminus H$. 

\begin{definition}\label{def 3.6}
Let $H\subset A$ be a compact open subgroup and let $\phi:H\to\widehat{H}$ be  a continuous symmetric homomorphism. 
 A function $h:A\to\mathbb{C}$ is said {\rm subcharacter of second degree} associated to $(H,\phi)$ if its restriction $h|_H$ is a character of second degree of $H$ associated to $\phi$ and $h(x)=0$ for $x\in A\setminus H$. 
\end{definition}
The following result provides the ambiguity function of a subcharacter of second degree. 
\begin{proposition}\label{pro 4}
 Let $h:A\to\mathbb{C}$ be a subcharacter of second degree associated to $(H,\phi)$ (cf. Definition \ref{def 3.6}); hence $H\subset A$ is a compact open subgroup and $\phi:H\to\widehat{H}$ is a continuous symmetric homomorphism.  Then  
\begin{equation}\label{eq 4.2}
V_h h(x,\xi)=|H|\overline{h(-x)}\chi_G(x,\xi)\qquad (x,\xi)\in\AhA,
\end{equation}
where $G$ is the maximal compact open isotropic subgroup of $\AhA$ corresponding to the triple $(H,H^\bot,\phi)$ (cf. Propositions \ref{pro 1}, \ref{pro 2} and \ref{pro 3}).

Moreover the function $|H|^{-1}V_h h$ is a subcharacter of second degree of $\AhA$ associated to pair $(G,\phi')$, where $\phi':G\to\widehat{G}$ is the continuous symmetric homomorphism in \eqref{eq phi'}. 
\end{proposition}

\begin{proof}
Since $h|_H$ is a character of second degree of $H$ associated to $\phi$, for $x,y\in H$ we have
\begin{equation}\label{eq 4.3}
h(y-x)=h(y)h(-x)\overline{\langle x,\phi(y)\rangle}. 
\end{equation}
We now compute $V_h h$. Observe that, if $x\in A\setminus H$, $\xi\in\hA$, 
\[
V_h h(x,\xi)=\langle h, M_\xi T_x h\rangle_{L^2(A)}=0,\]
because $H\cap (x+H)=\emptyset $ in that case.

On the other hand, if $x\in H$, $\xi\in\hA$, by \eqref{eq 4.3}, 
\begin{align*}
V_h h(x,\xi)&=\int_H \overline{\langle y,\xi\rangle} h(y)\overline{h(y-x)} \, dx\\
&= \overline{h(-x)}\int_H\overline{\langle y,\xi\rangle} |h(y)|^2 \langle y,\phi(x)\rangle\, dy\\
&= \overline{h(-x)} \int_H\langle y,\eta-\xi\rangle\, dy,
\end{align*}
where $\eta$ is any element of the coset $\phi(x)\subset\hA$ (recall $\phi:H\to\widehat{H}\simeq \hA/H^\bot$). Now, in the last integral the measure $dy$ can be regarded as a Haar measure of the compact {\it open} subgroup $H$. Therefore, the integral does not vanish if and only if $\xi$ and $\eta$ induce the same character of $H$ (\cite[Lemma 23.19]{hewitt63}), namely if $\eta-\xi\in H^\bot$, that is $\xi\in \phi(x)$. This proves \eqref{eq 4.2}. 

From \eqref{eq 4.2} we see that
\[
V_h h(x,\xi)= |H|\, \overline{h(-x)} \qquad (x,\xi)\in G.
\]
It is easy to see that the function $h(-x)$ for $x\in H$, is still a character of second degree of $H$ associated to the same homomorphism $\phi$. Hence, it follows from Proposition \ref{lem} that the function $|H|^{-1} V_h h(x,\xi)$, restricted to $G$, that is $\overline{h(-x)}$, is a character of second degree associated to the homomorphism $\phi'$ in \eqref{eq phi'}.
\end{proof}
We can now state the main result of this section.
\begin{theorem}\label{mainthm 1}
    Let $f\in L^2(A)$ and let $G=\{z\in\AhA:\ V_f f(z)\not=0\}$. The following statements are equivalent:
    \begin{itemize}
        \item[(a)] $|G|=1$.

        \par\medskip

        \item[(b)] $G$ is a maximal compact open isotropic subgroup of $\AhA$. \par\medskip
        \item[(c)]
         There exist $c\in\mathbb{C}\setminus\{0\}$, $x_0\in A$ and a subcharacter $h$ of second degree of $A$ such that $f=c T_{x_0} h$.
\end{itemize}
If one of the above condition is satisfied, the function $\|f\|_{L^2(A)}^{-2} V_f f$ is a subcharater of second degree of $\AhA$ associated to $(G,\phi')$, where $\phi':G\to\widehat{G}$ is given in \eqref{eq phi'}.
\end{theorem}

\begin{proof}
We can assume, without loss of generality, that $\|f\|_{L^2(A)}=1$.

(a) $\Rightarrow$ (b) By Theorem \ref{thm 2.2} we have $|V_f f|=\chi_G$, and $G$ is a compact open subset of $\AhA$. By Proposition \ref{pro 4.1}, $G$ is an isotropic subgroup of $\AhA$. Since $|G|=1$, it is maximal by Proposition \ref{pro 3}. \par\medskip
(b) $\Rightarrow$ (c) 
 By Proposition \ref{pro 3} we have that $|G|=1$ and therefore $|V_f f|=\chi_G$ by Theorem \ref{thm 2.2}. In fact, by Proposition \ref{pro 4.1}, the restriction of $V_f f$ to $G$ is a character of second degree of $\AhA$ associated to the homomorphism $\phi':G\to\widehat{G}$ in \eqref{eq phi'}. 

Let now $(H,H^\bot,\phi)$ be the triple associated to $G$ (cf. Propositions \ref{pro 1}, \ref{pro 2} and \ref{pro 3}), and let $h$ be a subcharacter of second degree of $A$ associated to $(H,\phi)$ (cf. Definition \ref{def 3.6}), which exists by Theorem \ref{thm reiter}. We know from Proposition \ref{pro 4} that the function $|H|^{-1} V_h h$, restricted to $G$, is a character of second degree associated to the same homomorphism $\phi'$ as above. Hence by Theorem \ref{thm reiter} there exists a character $g$ of $G$ such that 
\[
  V_f f(x,\xi) =g(x,\xi) |H|^{-1} V_h h(x,\xi)\qquad (x,\xi)\in G. 
\]
The character $g$ extends to a character of $\AhA$ (\cite[Corollary 24.12]{hewitt63}) and therefore there exist $y\in A$, $\eta\in\hA$ such that 
\[
g(x,\xi)=\langle x,\eta\rangle\overline{\langle y,\xi\rangle}.  
\]
We deduce that 
\[
V_f f(x,\xi) =|H|^{-1} \langle x,\eta\rangle\overline{\langle y,\xi\rangle}  V_h h(x,\xi)\qquad (x,\xi)\in G. 
\]
In fact this formula holds for every $(x,\xi)\in \AhA$ because both sides vanish on $\AhA\setminus G$, by  \eqref{eq 4.2} and the fact that $|V_f f|=\chi_G$. By a comparison with \eqref{eq 2.4} we deduce that 
\[
V_f f=|H|^{-1}V_{\pi(y,\eta)h}(\pi(y,\eta)h).
\]
By Proposition \ref{pro 2.3} we obtain 
\[
f=c|H|^{-1/2}\pi(y,\eta)h
\]
for some $c\in\mathbb{C}$, $|c|=1$. 

Setting $h':=M_\eta h$, we have $f=c' |H|^{-1/2} T_y h'$, $|c'|=1$, and $h'$ is a subcharacter of second degree associated to $(H,\phi)$, which gives the desired conclusion. \par\medskip
(c) $\Rightarrow$ (a) This is clear by Propositions \ref{pro 4} and \ref{pro 3}. \par\medskip

The last part of the statement is also clear by Proposition \ref{pro 4.1}.
\end{proof}

\section{Optimizers for the short-time Fourier transform}
In this section  we identify the functions $f,g\in L^2(A)$ such that 
\[
|\{z\in\AhA:\ V_g f(z)\not=0\}|=1.
\]
The following result will reduce the problem to the case $f=g$, that we addressed in the previous section. 
\begin{proposition}\label{pro 5.1}
    Let $f,g\in L^2(A)$, with $\|f\|_{L^2(A)}=\|g\|_{L^2(A)}=1$. Let $S=\{z\in\AhA:\ |V_g f(z)|=1\}$ and $G= \{z\in\AhA:\ |V_g g(z)|=1\}$. Let $z_0\in S$. Then $f=c\pi(z_0)g$ for some $c\in\mathbb{C}$, $|c|=1$, and $S=z_0+G$.
\end{proposition}
\begin{proof}
Since $|\langle f,\pi(z_0)g\rangle| = |V_g f(z_0)|=1$, we have $f=c\pi(z_0)g$ for some $c\in\mathbb{C}$, $|c|=1$. Hence, if $z\in A$, 
\[
|V_gf(z)|=|\langle \pi(z_0)g,\pi(z)g\rangle| =| V_g g(z-z_0)|,
\]
which implies $S=z_0+G$. 
\end{proof}
We therefore obtain the following characterization.
\begin{theorem}\label{mainthm 2} 
    Let $f,g\in L^2(A)$ and let $S=\{z\in\AhA:\ V_g f(z)\not=0\}$. The following statements are equivalent:
    \begin{itemize}
        \item[(a)] $|S|=1$.

        \par\medskip

        \item[(b)] $S$ is a coset in $\AhA$ of a maximal compact open isotropic subgroup. \par\medskip
        \item[(c)]
         There exist $c_1,c_2\in\mathbb{C}\setminus\{0\}$, $z_1,z_2\in \AhA$ and a subcharacter $h$ of second degree of $A$ such that $f=c_1\pi(z_1) h$ and $g=c_2\pi(z_2)h$.
\end{itemize} 
\end{theorem}
\begin{proof}
    The result easily follows from Theorem \ref{thm 2.2}, Proposition \ref{pro 5.1} and Theorem \ref{mainthm 1}. 
\end{proof}

\section{maximally localized Gabor orthonormal bases}
We recall that a Gabor orthonormal basis of $L^2(A)$ is an orthonormal basis of the form $ \mathcal{G}(f,\Gamma)$ (cf. \eqref{eq ort}), where $f\in L^2(A)$ and $\Gamma$ is a (not necessarily countable) subset of $\AhA$. 

The following result characterizes the Gabor orthonormal bases with $f$ maximally localized, in the sense that  the subset 
\[
G_f:=\{z\in \AhA:\ V_f f(z)\not=0\}
\]
has measure $1$. Observe that, by \eqref{eq 2.4},  $G_f=G_{\pi(w)f}$ for every $w\in\AhA$, so that all the elements of the basis are then maximally localized.  

\begin{theorem}\label{mainthm 3}
Let $f\in L^2(A)$, $\|f\|_{L^2(A)}=1$, with $|G_f|=1$; hence $G_f$ is a maximal compact open isotropic subgroup of $\AhA$ (by Theorem \ref{mainthm 1}). Let $\Gamma\subset\AhA$.

$\mathcal{G}(f,\Gamma)$ is an orthonormal basis of $L^2(A)$ if and only if $\Gamma$ is a set of representatives of the cosets of $G_f$ in $\AhA$. 
 \end{theorem}

 \begin{proof}
We know from Theorem \ref{mainthm 1} that $|V_f f|=  \chi_{G_f}$.

  Let $\Gamma$ be a set of representatives of the cosets of $G_f$ in $\AhA$.  
  Since \[
  |\langle \pi(z) f,\pi(w)f\rangle|=|V_f f(w-z)|,
  \]
  we see that  $\pi(z)f$ and $\pi(w)f$ are orthogonal if $z,w\in\Gamma$, $z\not=w$, because $\Gamma$ contains at most (in fact exactly) one element of each coset of $G_f$. 
  
  Let us verify that the set $\mathcal{G}(f,\Gamma)$  is also complete. We claim that 
\[
{\rm span}\big(\{\pi(z)f:\ z\in\Gamma\}\big)= {\rm span}\big(\{\pi(z)f:\ z\in\AhA\}\big).
\]
To see this, observe that if $z\in\AhA$ there exists $w\in\Gamma$ such that $z-w\in G_f$, hence $|V_f f(z-w)|=1$, which means that $\pi(z)f=c \pi(w)f$ for some $c\in\mathbb{C}$, $|c|=1$, which yields the claim. 

Now, the set $\{\pi(z)f:\ z\in\AhA\}$ is clearly complete, because if $g\in L^2(A)$ and $\langle g,\pi(z)f\rangle_{L^2(A)}=V_f g(z)=0$ for every $z\in\AhA$ then $g=0$, being the short-time Fourier transform injective (cf. \eqref{eq energy}).

Vice versa, suppose that $\mathcal{G}(f,\Gamma)$ is an orthonormal basis. If $z,w\in\Gamma$, $z\not=w$, since $\pi(z) f$ and $\pi(w)f$ are orthogonal, we have $V_f f(z-w)=0$, namely $z-w\not \in G$, i.e. $z$ and $w$ belong to different cosets.  Moreover if $\Gamma$ did not contain any element of some coset $z_0+G$, then the function $\pi(z_0)f$ would be orthogonal to all the functions $\pi(z)f$, with $z\in\Gamma$, which is impossible, since $\mathcal{G}(f,\Gamma)$  is a complete set by assumption. 
 \end{proof}
 Combining Theorems \ref{mainthm 1} and \ref{mainthm 3} we deduce the desired characterization of the maximally localized Gabor orthonormal basis. 
 \begin{corollary}\label{mainthm 4}
     Let $f\in L^2(A)$, $\|f\|_{L^2(A)}=1$, and $\Gamma\subset \AhA$. 
     
    Then $\mathcal{G}(f,
    \Gamma)$ is an orthonormal basis of $L^2(A)$ and $|G_f|=1$ if and only if $f=cT_{x_0} h$ for some  $c\in\mathbb{C}\setminus\{0\}$, $x_0\in A$ and some subcharacter $h$ of second degree of $A$,
      and $\Gamma$ is a set of representatives of the cosets of $G_f$ in $\AhA$. 
 \end{corollary}
 \begin{remark}\label{rem 6.3}
Observe that, in Corollary \ref{mainthm 4}, the sets $\{V_f(\pi(z)f)\not=0\}=z+G_f$, $z\in\Gamma$ (cf. \ref{eq 2.3}), define a tiling of $\AhA$ and $|z+G_f|=|G_f|=1$. Vice versa, if $G\subset\AhA$ is a maximal compact open isotropic subgroup (hence $|G|=1$), $(H,H^\bot,\phi)$ is the triple associated to $G$ (cf. Proposition \ref{pro 3}) and $h$ is a subcharacter associated to $(H,\phi)$, the function $f=|H|^{-1/2}h$ generates a Gabor orthonormal basis corresponding to the tiling generated by $G$.  
 \end{remark}

\begin{example}
Let $N\geq 1$ be an integer and let  $\mathbb{Z}_N=\mathbb{Z}/N\mathbb{Z}_N$ be the cyclic group of order $N$, equipped with the counting measure. As Haar measure on the dual group, we coherently choose the counting measure multiplied by $N^{-1}$. 

On $\mathbb{Z}_N$ a subcharacter of second degree has the form $h=M_\xi h_{b,p}$, where $\xi\in \widehat{\mathbb{Z}}_N$, $b\geq1$ is a divisor of $N$, $p\in\{0,\ldots,b-1\}$, and,  
\[
h_{b,p}(x)=\begin{cases}
 \exp\big(\frac{\pi i p x^2 b(1+b)}{N^2} \big)& x\in a\mathbb{Z}_N\\
0& x\in  \mathbb{Z}_N\setminus a\mathbb{Z}_N,
\end{cases}
\]
where $a=N/b$ (see \cite[Remark 2.1]{nicola23} and \cite[Section 3 (iii)] {feichtinger2}). 
We also have 
\[
G_h=\{(ma,nb+mp):\ m=0,\ldots,b-1,\ n=0,\ldots,a-1\}, 
\]
(see the proof of \cite[Theorem 1.2]{nicola23}), which is indeed a subgroup of $\mathbb{Z}_N\times \widehat{\mathbb{Z}}_N$ of cardinality $N$, hence of measure $1$ (incidentally, all the subgroups of cardinality $N$ have this form), and Corollary \ref{mainthm 4} applies.  

Of course, on $\mathbb{Z}_N$ there are Gabor orthonormal bases $\mathcal{G}(f,
\Gamma)$ which are not maximally localized, e.g.  we can take $f=2^{-1/2}\chi_{\{0,1\}}$ and $\Gamma=2\mathbb{Z}_N\times (N/2)\mathbb{Z}_N$, assuming $N\geq 4$ even. A straightforward computation shows that 
\[
G_f=(\{0,1,N-1\}\times\mathbb{Z}_N)\setminus\{(0,N/2)\},
\]
so that $|G_f|=3-1/N>1$. 
\end{example}
We also obtain the following result for finite Abelian groups.

\begin{corollary}
Let $A$ be a finite Abelian group and $S\subset \AhA$. The following statements are equivalent, for the family of operators $\{\pi(z):\ z\in S\}$:
\begin{itemize}
    \item[(a)] There exists a common eigenfunction. \par\medskip
    \item[(b)] The operators $\pi(z)$, $z\in S$, commute.  \par\medskip
    \item[(c)] There is a Gabor orthonormal basis, which consists of common eigenfunctions, generated by a function $f\in L^2(A)$, with $|G_f|=1$.   
\end{itemize}   
\end{corollary}
\begin{proof}
(a)$\Rightarrow$(b) If $f\in L^2(A)$ is a common eigenfunction, with $\|f\|_{L^2(A)}=1$, then $|V_f f(z)|=1$ for $z\in S$, because the eigenvalues of $\pi(z)$ have modulus $1$. Hence $S\subset G':=\{z\in\AhA:\ |V_f f(z)|=1\}$, and $G'$ is an {\it isotropic} subgroup of $\AhA$ by Proposition \ref{pro 4.1}. Hence the operators $\pi(z)$, with $z\in S$, commute by \eqref{eq 2.2}. \par\medskip(b)$\Rightarrow$(c)
Since the operators $\pi(z)$, $z\in S$, commute, the subgroup generated by $S$ is isotropic. It is moreover contained in some maximal isotropic subgroup $G$ (whose existence is obvious, because $A$ is finite; see also Corollary \ref{rem max}). Let $(H,H^\bot,\phi)$ be the triple associated to $G$ (cf. Propositions \ref{pro 1}, \ref{pro 2} and \ref{pro 3}) and let $h$ be a subcharacter of second degree associated to the pair $(H,\phi)$ (cf. Definition \ref{def 3.6}), which exists by Theorem \ref{thm reiter} (see also \cite{kaiblinger09}). By Proposition \ref{pro 4}, for the function $f=|H|^{-1/2}h$ we have $\|f\|_{L^2(A)}=1$ and $|V_f f|=\chi_G$, and by Theorem \ref{mainthm 3} $f$ generates a Gabor orthonormal basis $\mathcal{G}(f,\Gamma)$, for a suitable subset $\Gamma\subset\AhA$. Since $S\subset G$, $|\langle f,\pi(z)f\rangle|=|V_f f(z)|=1$ for $z\in S$, so that $f$ is a common eigenfunction of the operators $\pi(z)$, $z\in S$, and therefore, by \eqref{eq 2.2}, every function $\pi(w) f$, with $w\in \AhA$, is a common eigenfunction too.

\par\medskip(c)$\Rightarrow$(a) This is obvious. 
\end{proof}
We point out that extensive numerical experiments on eigenfunctions of time-frequency shifts were done by H. Feichtinger (personal communication), in connection with the work \cite{kaiblinger}.
\section{Lieb's uncertainty inequality}\label{sec conc}
 The following result was first proved in \cite{lieb} for the group $A=\R$, and then extended to a general LCA group in \cite{grochenig98}, following essentially the same proof.  

 We recall that every locally compact Abelian group $A$ is topologically isomorphic to $\R^d\times A_0$, for some integer $d\geq0$ and some LCA group $A_0$ containing a compact open subgroup, and the dimension $d$ is an invariant \cite[Theorem 24.30]{hewitt63}.

\begin{theorem}[Lieb's uncertainty inequality]\label{thm lieb} For $f,g\in L^2(A)$, we have 
\begin{equation}\label{eq lieb00}
\|V_g f\|_{L^p(\AhA)}\leq \Big(\frac{2}{p}\Big)^{d/p}\|f\|_{L^2(A)}\|g\|_{L^2(A)}\qquad 2\leq p<\infty
\end{equation}
and 
\begin{equation}\label{eq lieb01}
\|V_g f\|_{L^p(\AhA)}\geq \Big(\frac{2}{p}\Big)^{d/p}\|f\|_{L^2(A)}\|g\|_{L^2(A)}\qquad 1\leq p\leq 2.
\end{equation}
\end{theorem}
Using only \eqref{eq cs} one easily obtains similar estimates -- in fact weaker, if $d>1$ -- with the constant $\big(\frac{2}{p}\big)^{d/p}$ replaced by $1$ (see Theorem \ref{mainthm lieb}  below), namely
\begin{equation}\label {eq lieb1}
\|V_g f\|_{L^p(\AhA)}\leq \\|f\|_{L^2(A)}\|g\|_{L^2(A)}\qquad 2\leq p<\infty
\end{equation}
and 
\begin{equation}\label{eq lieb2}
\|V_g f\|_{L^p(\AhA)}\geq \|f\|_{L^2(A)}\|g\|_{L^2(A)}\qquad 0< p\leq 2,
\end{equation}
where now the case $0<p<1$ is also included. These estimates are sharp if $A$ contains a compact open subgroup (i.e. $d=0$).

We are going to prove that the pairs of functions $f,g$ for which equality is attained in \eqref{eq lieb1} or \eqref{eq lieb2} are precisely those for which the set where $V_g f\not=0$ has measure $1$, which are characterized in Theorem \ref{mainthm 2}.
\begin{theorem}\label{mainthm lieb}
Let $A$ be any LCA group. Then \eqref{eq lieb1} and \eqref{eq lieb2} hold true. 

Equality holds in \eqref{eq lieb1} for some $p\in (2,\infty)$ and $f,g\in L^2(A)\setminus\{0\}$ if and only if there exist $c_1,c_2\in\mathbb{C}\setminus\{0\}$, $z_1,z_2\in \AhA$ and a subcharacter $h$ of second degree of $A$ such that $f=c_1\pi(z_1) h$ and $g=c_2\pi(z_2)h$. In that case, equality occurs in \eqref{eq lieb1} for every $p\in[2,\infty)$.

A similar uniqueness result holds true for the inequality \eqref{eq lieb2}, for $0<p<2$.
\end{theorem}
\begin{proof}
 We can suppose $\|f\|_{L^2(A)}=\|g\|_{L^2(A)}=1$. 
 
 Let $2\leq p<\infty$ and set $S=\{z\in\AhA:\ V_g f(z)\not=0\}$.  Using \eqref{eq cs} and \eqref{eq energy} we see that 
\begin{align*}
\int_{S} |V_g f(x,\xi)|^p\, dx\, d\xi &= \int_{S}  |V_g f(x,\xi)|^{p-2} |V_g f(x,\xi)|^2\, dx\, d\xi\\
&\leq \int_{S}  |V_g f(x,\xi)|^2\, dx\, d\xi=1, 
\end{align*}
which proves \eqref{eq lieb1}. If equality occurs in the above inequality and $2<p<\infty$ then $|V_g f|=\chi_S$ and $|S|=1$, which implies the desired conclusion for the functions $f$ and $g$ by Theorem \ref{mainthm 2}.

The result for the inequality \eqref{eq lieb2}, hence $0<p\leq 2$, is analogous, using 
\begin{align*}
1=\int_{S} |V_g f(x,\xi)|^2\, dx\, d\xi &= \int_{S}  |V_g f(x,\xi)|^{2-p} |V_g f(x,\xi)|^p\, dx\, d\xi\\
&\leq \int_{S}  |V_g f(x,\xi)|^p\, dx\, d\xi.
\end{align*}
\end{proof}
\begin{remark}\label{rem pto0} 
If $f,g\in L^2(A)$, we have $V_g f\in L^\infty(\AhA)$ by \eqref{eq cs}. Hence, by monotone convergence, 
\[
\lim_{p\to 0^+}\int_{\AhA} |V_g f(x,\xi)|^p\, dx d\xi= |\{z\in\AhA:\ V_g f(z)\not=0\}|. 
\]
As a consequence, raising to the power $p$ both sides of \eqref{eq lieb2} and taking the limit as $p\to 0^+$,  we obtain that, if $f$ and $g$ are non-zero, 
\[
|\{z\in\AhA:\ V_g f(z)\not=0\}|\geq 1,
\]
that is the inequality in Theorem \ref{thm 2.2}. 
%We also observe that, unlike \eqref{eq lieb2}, it was conjectured in \cite{lieb} that \eqref{eq lieb01} holds even for $0<p<1$, but this remains in ope issue for a general $g$, even in $\R^d$. For the window $g(x)=e^{-\pi|x|^2}$ in $\R^d$, the result was proved in \cite{}.    
\end{remark}
It is easy to check that, on a general LCA group $A=\R^d\times A_0$, for $f_1,g_1\in L^2(\R^d)$ and $f_2,g_2\in L^2(A_0)$, we have 
\[
V_{g_1\otimes g_2}(f_1\otimes f_2)= V_{g_1} f_1 \otimes V_{g_2} f_2. 
\]
As a consequence, for fixed $p\in [1,\infty)$, if $f_1,g_1$ is a pair of optimizers for Lieb's $L^p$-inequality in $\R^d$ (Theorem \ref{thm lieb}) and similarly for $f_2,g_2$ on $A_0$, then $f_1\otimes f_2,g_1\otimes g_2$ is a pair of optimizers  for the Lieb's $L^p$-inequality on $A$. 
We now show a family of such optimizers. To this end, we need some terminology, inspired by \cite{lieb}. 

\begin{definition}
A function $f$ on $\R^d$ a said a Gaussian if 
\[
f(x)=\exp(-\alpha x\cdot x+i\beta x\cdot x+\gamma\cdot x+\delta)
\]
where $\alpha$ is a real symmetric positive definite $d\times d$ matrix, $\beta$ is a real symmetric $d\times d$ matrix, $\gamma\in\mathbb{C}^d$ and $\delta\in\mathbb{C}$. 
Two functions $f,g$ are called a {\it matched Gaussian pair} if $f$ and $g$ are both Gaussians with the same $\alpha$'s and $\beta$'s but with possibly different $\gamma$'s and $\delta$'s.

Similarly, a pair of functions $f,g$ on a LCA group $A$ is said a {\it matched pair of subcharacters of second degree} if $f=c_1\pi(z_1) h$ and $g=c_2\pi(z_2)h$ for some $c_1,c_2\in\mathbb{C}\setminus\{0\}$, $z_1,z_2\in\AhA$ and some subcharacter $h$ of second degree of $A$. 

\end{definition} 
It is easy to check that for a matched Gaussian pair $f,g$, equality occurs in \eqref{eq lieb00} and \eqref{eq lieb01} ($A=\R^d$). For $A=\R$ it was proved in \cite{lieb} that these are in fact the only pairs of non-zero optimizers if $p\not=2$.

 The previous discussion therefore leads  to the following result.  
\begin{proposition}\label{pro 7.3}
    Let $f_1,g_1$ be a matched Gaussian pair on $\R^d$ and let $f_2,g_2$ be a matched  pair of subcharacters of second degree on $A_0$. Then, for the functions $f:=f_1\otimes f_2$ and $g:=g_1\otimes g_2$ on $A=\R^d\times A_0$, equality occurs in \eqref{eq lieb00} and \eqref{eq lieb01} for every $p\in [1,\infty)$.  
\end{proposition}

The optimizers where $f_1$ and $g_1$ are time-frequency shifts of the Gaussian $\exp(-\pi |x|^2)$, and $f_2$ and $g_2$ are time-frequency shifts of the characteristic function of some compact open subgroup of $A$, were already known from \cite{grochenig98}. 

We postpone to a subsequent work the problem of identifying all the optimizers on a general LCA group -- as already observed, the case $A=\R$ was already addressed in \cite{lieb}, whereas the case $A=A_0$ is the content of Theorem \ref{mainthm lieb} above.  

\section*{Acknowledgments}
We would like to thank Hans Feichtinger for enlightening discussions.

%\bibliographystyle{abbrv} 
%\bibliography{biblio}

\end{document}